\documentclass{article}

\usepackage{amsmath,amsthm,amssymb,amsfonts}
\usepackage{amscd}
\usepackage{verbatim}

\usepackage{hyperref}

\usepackage{tikz}
\usepackage{tikz-cd}

\usepackage{amsmath, amsfonts,amssymb}

\usepackage{stmaryrd} 

\usepackage{amsthm}
\theoremstyle{plain}

\newtheorem{thm}{Theorem}
\newtheorem*{thm*}{Theorem}

\newtheorem*{lem*}{Lemma}
\newtheorem{lem}[thm]{Lemma}
\newtheorem{cor}[thm]{Corollary}
\newtheorem*{cor*}{Corollary}
\newtheorem{prop}[thm]{Proposition}

\newtheorem*{conj*}{Conjecture}

\theoremstyle{definition}
\newtheorem{exl}[thm]{Example}
\newtheorem*{exl*}{Example}

\newtheorem{defn}[thm]{Definition}

\newtheorem*{note*}{Note}

\newcommand{\R}{\mathbb{R}}
\newcommand{\Z}{\mathbb{Z}}

\newcommand{\N}{\mathbb{N}}

\renewcommand{\phi}{\varphi}
\newcommand{\RT}{\text{\textit{RT}}}
\newcommand{\Reid}{\mathcal{R}}
\DeclareMathOperator{\Fix}{Fix}
\DeclareMathOperator{\ind}{ind}
\DeclareMathOperator{\id}{id}
\DeclareMathOperator{\tr}{tr}
\newcommand{\pd}{\partial}

\newcommand{\ep}{\epsilon}
\newcommand{\adm}{\mathcal{C}}

\begin{document}
\title{The Reidemeister Trace of an $n$-valued map}
\author{P. Christopher Staecker}

\maketitle

\begin{abstract}
In topological fixed point theory, the Reidemeister trace is an invariant associated to a selfmap of a polyhedron which combines information from the Lefschetz and Nielsen numbers. In this paper we define the Reidemeister trace in the context of $n$-valued selfmaps of compact polyhedra. We prove several properties of the Reidemeister trace which generalize properties from the single-valued theory, and prove an averaging formula.
\end{abstract}

\section{Introduction}
Given sets $X$ and $Y$ and a positive integer $n$, an $n$-valued function from $X$ to $Y$ is a set-valued function $f$ on $X$ such that $f(x)\subseteq Y$ has cardinality exactly $n$ for every $x\in X$. Equivalently, an $n$-valued function on $X$ is a single-valued function $f:X \to D_n(Y)$, where $D_n(Y)$ is the \emph{unordered configuration space of $n$ points in $Y$}, defined as:
\[ D_n(Y) = \{ \{y_1,\dots,y_n\} \mid y_i \in Y, y_i \neq y_j \text{ for } i\neq j \}. \]

When $Y$ is a topological space, we give $D_n(Y)$ a topology as follows: begin with the product topology on $Y^n$, then consider the subspace $F_n(Y)$ of tuples $(y_1,\dots,y_n)\in Y^n$ with $y_i \neq y_j$ for $i\neq j$. This is the \emph{ordered configuration space}. Then $D_n(Y)$ is the quotient of $F_n(Y)$ up to ordering, and so its topology is given by the quotient topology. When $f:X\to D_n(Y)$ is continuous, we call it an \emph{$n$-valued map from $X$ to $Y$}. Continuity of $n$-valued maps can also be defined in terms of lower- and upper-semicontinuity. These approaches are equivalent: see \cite{bg18}.

In this paper all spaces will be assumed to be finite connected polyhedra, and we will generally discuss $n$-valued ``selfmaps'' of the form $f:X \to D_n(X)$. A \emph{fixed point} of such a map is a point $x\in X$ with $x \in f(x)$, and we denote the set of fixed points of $f$ by $\Fix(f)\subset X$. Topological fixed point theory of $n$-valued maps was first studied by Schirmer in \cite{schi84a, schi84b}. 

Nielsen theory (see \cite{jian83}) of a single-valued map partitions the fixed point set into fixed point classes, which are indexed by lifting classes, also called Reidemeister classes. The fixed point index assigns an integer value to each fixed point class, and summing the fixed point index over the entire domain will give the Lefschetz number, which can also be computed as an alternating sum of traces in homology. 

The Reidemeister trace is a single invariant which combines the information from the Lefshetz and Nielsen numbers. It was defined by Wecken \cite{weck41} using a trace formula which incorporates information about the algebraic lifting classes. A more modern presentation of the Reidemeister trace was given in \cite{huss82}. A more readable presentation appears in \cite{geog02}. Our goal is to generalize the definition and basic results for the Reidemeister trace to the setting of $n$-valued maps.

We will briefly review the theory of Reidemeister classes for $n$-valued maps, which was defined in \cite{bdds20}. 

For a finite connected polyhedron $X$, we always denote the fundamental group of $X$ by $\pi = \pi_1(X)$. Given the universal covering $p:\tilde X \to X$, the \emph{orbit configuration space} with respect to this cover is:
\[ F_n(\tilde X, \pi) = \{ (\tilde x_1,\dots,\tilde x_n) \in \tilde X^n \mid p(\tilde x_i) \neq p(\tilde x_j) \text{ for $i \neq j$} \}, \]
and the map $p^n:F_n(\tilde X,\pi) \to D_n(X)$ given by applying $p$ to each coordinate is a covering map.

Given a map $f:X\to D_n(X)$, we choose a lifting $\tilde f: \tilde X \to F_n(\tilde X,\pi)$ so that the diagram commutes:
\[ 
\begin{tikzcd}
\tilde X \arrow[r,"\tilde f"] \arrow[d,"p"] & F_n(\tilde X,\pi) \arrow[d,"p^n"] \\
X \arrow[r,"f"] & D_n(X)
\end{tikzcd}
\]
Since $F_n(\tilde X,\pi) \subset \tilde X^n$ we can write $\tilde f = (\tilde f_1,\dots, \tilde f_n)$ for single-valued maps $\tilde f_k: \tilde X \to \tilde X$. 

The covering group of $F_n(\tilde X,\pi)$ is the semidirect product $\pi^n \rtimes \Sigma_n$, where $\pi^n$ is the $n$-fold cartesian product of $\pi$, and $\Sigma_n$ is the symmetric group on $n$ elements. The action of $\pi^n \rtimes \Sigma_n$ on $F_n(\tilde X, \pi)$ is given by:
\[ (\alpha_1,\dots,\alpha_n; \sigma)\cdot (\tilde x_1,\dots, \tilde x_n) = (\alpha_1 \tilde x_{\sigma^{-1}(1)}, \dots, \alpha_n \tilde x_{\sigma^{-1}(n)}). \]

The group operation and inverse for the semidirect 
product take the form:
\begin{align*} 
(\alpha_1,\dots,\alpha_n; \sigma)(\beta_1,\dots,
\beta_n;\rho) &= (\alpha_1\beta_{\sigma^{-1}(1)}, \dots, 
\alpha_n\beta_{\sigma^{-1}(n)}; \sigma \circ \rho) \\
(\alpha_1,\dots,\alpha_n;\sigma)^{-1} &= 
(\alpha^{-1}_{\sigma(1)},\dots, 
\alpha^{-1}_{\sigma(n)}; \sigma^{-1}).
\end{align*}

Any pair $(\alpha,k)\in \pi \times \{1,\dots,n\}$ determines the fixed point class $p\Fix(\alpha^{-1}\tilde f_k)$. 
The map $f$ and its lifting $\tilde f$ induce a homomorphism $\tilde f_\#: \pi \to \pi^n \rtimes \Sigma_n$, defined by the formula
\begin{equation}\label{inducedeq} 
\tilde f \circ \gamma = \tilde f_\# (\gamma) \circ \tilde f. 
\end{equation}
We can write $\tilde f_\#$ in coordinates as $\tilde f_\#= (\phi_1,\dots,\phi_n;\sigma)$, where $\phi_k:\pi \to \pi$ are functions and $\sigma: \pi \to \Sigma_n$ is a homomorphism. Although $\tilde f_\#$ and $\sigma$ are homomorphisms, the $\phi_k$ may not be homomorphisms. Specifically, they obey the following formulas:

\begin{prop}
As above, let $\tilde f_\# = (\phi_1,\dots,\phi_n;\sigma)$, and let $x,y\in \pi$. Then:
\begin{align*} 
\phi_k(xy) &= \phi_{k}(x)\phi_{\sigma^{-1}_x(k)}(y) \\
\phi_k(x^{-1}) &= \phi_{\sigma_x(k)}(x)^{-1}
\end{align*}
\end{prop}
\begin{proof}
Lemma 2.5 of \cite{bdds20} shows that $\tilde f_\#:\pi \to \pi^n\rtimes \Sigma_n$ is a homomorphism. Thus we can compute:
\begin{align*} 
\tilde f_\#(xy) &= \tilde f_\#(x) \tilde f_\#(y) = (\phi_1(x),\dots,\phi_n(x); \sigma_x)(\phi_1(y),\dots,\phi_n(y);\sigma_y) \\
&= (\phi_1(x)\phi_{\sigma_x^{-1}(1)}(y), \dots, \phi_n(x)\phi_{\sigma_x^{-1}(n)}(y); \sigma_x \circ \sigma_y)
\end{align*}
and reading coordinate $k$ above gives the first desired formula.

For the second formula, letting $y=x^{-1}$ in the first formula gives:
\[ 1 = \phi_k(x) \phi_{\sigma^{-1}_{x}(k)}(x^{-1}) \]
and so $\phi_k(x)^{-1} = \phi_{\sigma^{-1}_{x}(k)}(x^{-1})$ which (after reindexing) gives the second formula.
\end{proof}

There is an algebraic Reidemeister equivalence relation on $\pi \times \{1,\dots,n\}$ defined as follows: 
\begin{equation}\label{tceq}
(\alpha,k) \sim (\beta,j) \iff \exists \gamma\in \pi \text{ with } \sigma_\gamma(k)=j \text{ and } \alpha = \phi_j(\gamma)^{-1} \beta \gamma 
\end{equation}
where we have written $\sigma(\gamma)\in \Sigma_n$ as $\sigma_\gamma$. 

The equivalence classes with respect to this relation form the set of \emph{Reidemeister classes}, denoted $\Reid(\tilde f_\#) = \pi\times \{1,\dots,n\} / \sim$. The class represented by $(\alpha,k)$ will be written as $[(\alpha,k)]_{\tilde f_\#}$, or simply as $[(\alpha,k)]$ when the context is clear.

\begin{prop}(\cite{bdds20}, proof of Theorem 2.6)
With the notations above, for any $\gamma\in \pi$ we have:
\[ \tilde f_k(\gamma \tilde x) = \phi_k(\gamma) \tilde f_{\sigma_\gamma^{-1}(k)}(\tilde x). \]
\end{prop}
\begin{proof}
Writing \eqref{inducedeq} in coordinates gives:
\begin{align*}
(\tilde f_1,\dots,\tilde f_n)\circ \gamma &= \tilde f_\#(\gamma) \circ \tilde f \\
&= (\phi_1(\gamma),\dots,\phi_n(\gamma); \sigma_\gamma) (\tilde f_1,\dots,\tilde f_n) \\
&= (\phi_1(\gamma)\tilde f_{\sigma^{-1}_\gamma(1)}, \dots, \phi_n(\gamma)\tilde f_{\sigma^{-1}_\gamma(n)}),
\end{align*}
and reading coordinate $k$ above gives the result.
\end{proof}

\begin{note*}
The exposition in \cite{bdds20} instead lets $(\alpha,k)$ represent the fixed point class $p\Fix(\alpha \tilde f_k)$ (while we use $\alpha^{-1}\tilde f_k$ rather than $\alpha\tilde f_k$). This is merely a convention, but our use of $\alpha^{-1}$ is required in order for the trace formula of Theorem \ref{weckentracethm} to hold. 

This convention also effects the Reidemeister relation: in \cite{bdds20}, the relation uses $\sigma(j)=k$ and $\alpha = \gamma \beta \phi_j(\gamma^{-1})$. To accomodate the choice of $\alpha^{-1}\tilde f_k$ rather than $\alpha\tilde f_k$, the appropriate formula is instead given by \eqref{tceq}.

To demonstrate the consistency of these choices, we will briefly show that $p\Fix(\alpha^{-1}\tilde f_k) = p\Fix(\beta^{-1}\tilde f_j)$ implies \eqref{tceq}. (In fact it is equivalent to \eqref{tceq} and more is true, see \cite{bdds20}.)

Let $x\in p\Fix(\alpha^{-1}\tilde f_k)$. Then there is some $\tilde x$ with $p(\tilde x)=x$ and $\alpha^{-1}\tilde f_k(\tilde x) = \tilde x$. Since $p\Fix(\alpha^{-1}\tilde f_k) = p\Fix(\beta^{-1}\tilde f_j)$, there is some $\gamma \in \pi$ with $\gamma\tilde x \in p\Fix(\beta^{-1}\tilde f_j)$, that is, $\beta^{-1}\tilde f_{j}(\gamma \tilde x) = \gamma \tilde x$. Equivalently, we have 
\[ \gamma^{-1}\beta^{-1} \phi_j(\gamma)\tilde f_{\sigma^{-1}_\gamma(j)}(\tilde x) =\tilde x. \] 
Combining the above gives:
\[ \alpha^{-1}\tilde f_k(\tilde x) = \gamma^{-1}\beta^{-1} \phi_j(\gamma)\tilde f_{\sigma^{-1}_\gamma(j)}(\tilde x). \]
Thus the two maps $\alpha^{-1}\tilde f_k$ and $\gamma^{-1}\beta^{-1} \phi_j(\gamma)\tilde f_{\sigma^{-1}_\gamma(j)}$ agree at a point. But since $\tilde f$ maps into $F_n(\tilde X, \pi)$, this is only possible when these two maps are equal. Thus we must have $k=\sigma^{-1}_\gamma(j)$, and $\alpha^{-1}=\gamma^{-1}\beta^{-1}\phi_j(\gamma)$, which implies \eqref{tceq}. 
\end{note*}

The Reidemeister trace we define will be an element of the abelian group $\Z\Reid(\tilde f_\#)$ of formal integer sums of Reidemeister classes. 

We will also make use of the fixed point index of an $n$-valued map, first defined by Schirmer in \cite{schi84a}. Let $f:X\to D_n(X)$ be an $n$-valued map, and $U\subseteq X$ an open subset such that $\Fix(f)\cap U$ is compact. Then $\ind(f,U) \in \Z$ is defined, and obeys the following properties:
\begin{itemize}
\item (Additivity) If $\Fix(f)\cap U \subset U_1 \cup U_2$ with $U_1 \cap U_2 = \emptyset$, then 
\[ \ind(f,U) = \ind(f,U_1) + \ind(f,U_2). \]
\item (Excision) If $V$ is an open subset of $U$ and $\Fix(f,U) \subset V$, then $\ind(f,U) =\ind(f,V)$.
\item (Homotopy) If $f$ and $g$ are ``admissably homotopic'' on $U$ then $\ind(f,U) = \ind(g,U)$.
\item (Solution) If $\ind(f,U)\neq 0$, then $\Fix(f) \cap U \neq \emptyset$.
\end{itemize}
The main result of \cite{stae18} is that this index is uniquely determined by the additivity and homotopy properties, together with a normalization that sets the index of a constant to 1.

Schirmer's original definition of the index is in terms of splittings: if $f$ splits into single valued maps $f = \{f_1,\dots,f_n\}$ on $U$, and $x$ is an isolated fixed point of $f_k$ in $U$, then the $n$-valued index $\ind(f,U)$ is defined to be the single-valued index $\ind(f_k,U)$.

\mathchardef\myhyphen="2D

The theory of $n$-valued maps as developed by Schirmer and Brown is paralleled by work of Crabbe which also discusses fixed point invariants of $n$-valued maps. Crabbe's definition of an $n$-valued map is different, using finite covers, resulting in maps whose set of values have cardinality at most (but not always equal to) $n$. In \cite{crab15}, Crabbe defines an invariant $h\myhyphen L(f/p)$ which is analagous to our Reidemeister trace. Our presentation will attempt to mimic classical single-valued constructions whenever possible.

In the next section we define the Reidemeister trace of an $n$-valued map, and prove that it obeys a familiar trace formula. Section \ref{lemsec} is devoted to the proof of an algebraic lemma on the homotopy invariance of the trace formula. In Section \ref{propertiessec} we demonstrate some basic properties of the Reidemeister trace, and in Section \ref{localsec} we discuss a local version, which we use to derive an averaging formula.

\section{Definition of the Reidemeister trace}\label{defsec}
The Reidemeister trace of an $n$-valued map $f:X\to D_n(X)$ and a lifting $\tilde f: \tilde X \to F_n(\tilde X,\pi)$ is an element of $\Z\Reid(\tilde f_\#)$ in which each term is a Reidemeister class of $\tilde f$, with coefficient given by its fixed point index. We use a straightforward formal sum as the definition of the Reidemeister trace, and later in Theorem \ref{weckentracethm} we demonstrate that it can be computed as an alternating sum of traces of certain matrices.

\begin{defn}\label{RTdef}
Let $f:X\to D_n(X)$ be an $n$-valued map, and let $\tilde f: \tilde X \to F_n(\tilde X,\pi)$ be a lifting with $\tilde f = (\tilde f_1,\dots,\tilde f_n)$. Then we define the \emph{Reidemeister trace} of $f$ and $\tilde f$ as:
\[ \RT(f,\tilde f) = \sum_{[(\alpha,k)] \in \Reid(\tilde f_\#)} \ind(f,U_{[(\alpha,k)]}) [(\alpha,k)]_{\tilde f_\#} \in \Z\Reid(\tilde f_\#), \]
where each $U_{[(\alpha,k)]} \subset X$ is an open set containing $p\Fix(\alpha^{-1}\tilde f_k)$ and no other fixed points of $f$.
\end{defn}
By the excision property of the index, alternative choices of the sets $U_{[(\alpha,k)]}$ will give the same values of the index, so the definition above does not depend on the choice of sets $U_{[(\alpha,k)]}$.

The Reidemeister trace of an $n$-valued map is related to the Nielsen and Lefschetz numbers in the expected way:

\begin{thm}
Let $f:X\to D_n(X)$ be an $n$-valued map, with some lift $\tilde f:X\to F_n(X,\pi)$. The sum of coefficients in $\RT(f,\tilde f)$ is the Lefschetz number $L(f)$, and the number of nonzero terms is the Nielsen number $N(f)$.
\end{thm}
\begin{proof}
Let $c:\Z\Reid(\tilde f_\#) \to \Z$ be the sum-of-coefficients homomorphism. Then by Theorem \ref{weckentracethm} we have
\[ c(\RT(f,\tilde f)) = \sum_{[(\alpha,k)]\in \Reid(\tilde f_\#)} \ind(f,U_{[(\alpha,k)]}). \]
Since $\bigcup U_{[(\alpha,k)]}$ contains all fixed points, by the excision and additivity properties of the index we have $c(\RT(f,\tilde f)) = \ind(f,X) = L(f)$. 

For the statement about the Nielsen number, Definition \ref{RTdef} immediately implies that the number of nonzero terms of $\RT(f,\tilde f)$ is the number of fixed point classes having nonzero index, which is the definition of $N(f)$.
\end{proof}

Theorem 2.7 of \cite{bdds20} shows that if $f$ is homotopic to $g$, then $\tilde f_\# = g_\#$. Thus we may identify $\Z\Reid(\tilde f_\#)$ with $\Z\Reid(g_\#)$, and we have the familiar homotopy invariance property.

\begin{thm}\label{htp}
Let $f,g:X\to D_n(X)$ be homotopic by a homotopy $H:X\times I \to D_n(X)$. Let $\tilde f=(\tilde f_1,\dots,\tilde f_n)$ be a lifting of $f$, and let $\tilde g=(\tilde g_1,\dots,\tilde g_n)$ be the lifting of $g$ obtained by lifting the homotopy, starting at $\tilde f$. Then, identifying $\Z\Reid(\tilde f_\#)$ with $\Z\Reid(g_\#)$, we have:
\[ \RT(f,\tilde f) = \RT(g,\tilde g). \]
\end{thm}
\begin{proof}
In \cite{schi84b}, Schirmer discusses fixed point classes of $n$-valued maps and their behavior under homotopy. Lemma 6.2 of \cite{schi84b} shows that the homotopy $H$ gives a bijective correspondence between fixed point classes of $f$ and $g$. Lemma 6.4 of \cite{schi84b} shows that this correspondence preserves the index.

Let $C$ be a fixed point class of $f$, and let $D$ be the fixed point class corresponding to $C$ via the homotopy $H$. Then $\ind(f,U_C) = \ind(f,U_D)$. It remains to show that $[C]_{\tilde f_\#} = [D]_{g_\#}$. 

Let $[C]_{\tilde f_\#} = [(\alpha,k)]_{\tilde f_\#}$, so $C = p\Fix(\alpha^{-1} \tilde f_k)$. Since $D$ corresponds to $C$ via the homotopy, and the homotopy lifts to carry $\tilde f_k$ to $\tilde g_k$, we have $D = p \Fix(\alpha^{-1} \tilde g_k)$, and thus $[D]_{g_\#} = [C]_{\tilde f_\#}$ as desired.
\end{proof}

When the fixed point set of $f$ is finite, we can express the Reidemeister trace by summing over the fixed point set.

For an isolated fixed point $x\in \Fix(f)$ and lift $\tilde f$, let $[x] \in \Reid(\tilde f_\#)$ denote the Reidemeister class of the fixed point class containing $x$. Specifically, we will have $[x] = [(\alpha,k)]$ when $x \in p\Fix(\alpha^{-1}\tilde f_k)$, or equivalently, there is some $\tilde x\in \tilde X$ with $p(\tilde x)=x$ and $\tilde f_k(\tilde x) = \alpha\tilde x$.

\begin{thm}\label{regularwecken}
Let $f:X\to D_n(X)$ be an $n$-valued map with finite fixed point set, and let $\tilde f: \tilde X \to F_n(\tilde X,\pi)$ be a lifting. Then:
\[ \RT(f,\tilde f) = \sum_{x\in \Fix(f)} \ind(f,U_x) [x], \]
where $U_x$ is a neighborhood of $x$ containing no other fixed points of $f$.
\end{thm}
\begin{proof}
It suffices to show that, if $[(\alpha,k)]$ represents a fixed point class $\{x_1,\dots,x_m\}$, then $\ind(f,U_{[(\alpha,k)]})= \sum_{i=1}^m \ind(f,U_{x_i})[x_i]$. Because the only fixed points in $U_{[(\alpha,k)]}$ are the points $\{x_1,\dots,x_m\}$, this follows from the excision and additivity properties of the index.
\end{proof}

Now we will develop the formula for the Reidemeister trace as an alternating trace sum of certain matrices. 
The trace formula for $\RT(f,\tilde f)$ is only defined when $f$ is a simplicial map on a polyhedron.
In \cite{schi84a}, Schirmer defines the notion of an $n$-valued simplicial map between simplicial complexes. Let $K$ and $L$ be simplicial complexes with geometric realizations $|K|$ and $|L|$, and let $f\colon |K| \to D_n(|L|)$ be an $n$-valued map. Then $f$ is an \emph{$n$-valued simplicial map} when $f$ maps any simplex $\sigma$ of $K$ to a disjoint union of $n$ simplices of $L$. We call such a simplicial map \emph{regular} when all fixed points are in the interior of maximal simplices. If $f$ is regular and $x\in \Fix(f)$, then $f$ carries the simplex containing $x$ homeomorphically onto itself, and so there is a well-defined orientation $o(f,x)\in \{+1,-1\}$ of $f$ at $x$.


Let $C_q(K)$ denote the dimension $q$ chain group of a simplicial complex $K$. When $K'$ is a subdivision of $K$ and $q\in \N$, there is a natural homomorphism $\epsilon^q: C_q(K) \to C_q(K')$ which takes any $q$-simplex $\sigma \in K$ and sets $\epsilon^q(\sigma)$ equal to the sum of all $q$-simplices from $K'$ which arise from subdivision of $\sigma$. 

If $K'$ is a subdivision of $K$ and $f:K' \to K$ is simplicial, let $f^q:C_q(K')\to C_q(K)$ be the homomorphism induced by $f$. Then the composition $f^q\circ \epsilon^q$ is an endomorphism of $C_q(K)$. 

Letting $p:\tilde X \to X$ be the universal cover of $X$, any triangulation $K$ of $X$ induces a triangulation $\tilde K$ on $\tilde X$. We will label the simplices of $\tilde K$ as follows: for each simplex $\sigma\in K$, arbitrarily choose a lifting $\tilde \sigma \in \tilde K$. All other simplices of $\tilde K$ projecting to $\sigma$ can be described as $\gamma \tilde \sigma$ for some $\gamma\in \pi$. In this way, we view the chain group $C_q(\tilde K)$ as a $\Z\pi$-module, and any homomorphism $f:C_q(\tilde K) \to C_q(\tilde K)$ can be expressed as a matrix with entries in $\Z\pi$, and has a trace (diagonal sum) which we denote $\tr f\in \Z\pi$. 

Given an $n$-valued map $f$ and lifting $\tilde f$, let $\rho: \Z(\pi \times \{1,\dots,n\}) \to \Z\Reid(\tilde f_\#)$ be the natural quotient homomorphism. The following theorem is the justification for the term ``trace'' in describing $\RT(f,\tilde f)$. 
\begin{thm}\label{regweckentracethm}
Let $K$ be a complex with dimension $d$, let $K'$ be a subdivision of $K$, and let $f:|K'| \to D_n(|K|)$ be a regular $n$-valued simplicial map. Let $\tilde f = (\tilde f_1,\dots, \tilde f_n): |K'| \to F_n(|K|,\pi)$ be a lifting. Then:
\begin{equation}\label{traceeq}
\RT(f,\tilde f) = \sum_{q=0}^d \sum_{k=1}^n (-1)^q  \rho(\tr (\tilde f_k^q \circ \epsilon^q) \times \{k\}) \in \Z \Reid(\tilde f_\#) .
\end{equation}
Above we have abused notation slightly in the product with $\{k\}$. For a group ring element $s = \sum c_ia_i \in \Z\pi$, we write $s \times \{k\}$ to indicate the group ring element $\sum c_i (a_i,k) \in \Z(\pi \times \{1,\dots,n\})$. 
\end{thm}
\begin{proof}
It will be convenient to discuss the related trace:
\[ T(f,\tilde f) = \sum_{q=0}^d \sum_{k=1}^n (-1)^q  \tr (\tilde f_k^q \circ \epsilon^q) \times \{k\} \]
which is an element of $\Z(\pi \times \{1,\dots,n\})$. To prove the theorem, it will suffice to show that $\rho(T(f,\tilde f)) = \RT(f,\tilde f)$.

Any term contributing to the traces in $T(f,\tilde f)$ must arise from a basis simplex $\tilde \sigma$ satisfying $\tilde f_k(\tilde \sigma) = \alpha\tilde \sigma$ for some $\alpha\in \pi$. In this case, letting $\sigma = p(\tilde \sigma)$, we have $f(\sigma)=\sigma$. By Brouwer's fixed point theorem $f$ has a fixed point inside $\sigma$, and since $f$ is regular, there is only one fixed point inside $\sigma$. Conversely, any simplex $\sigma$ containing a fixed point will contribute some nonzero term to the trace.

Thus there is a correspondence between terms of $T(f,\tilde f)$ and simplices containing fixed points. In particular, for each fixed point $x$, let $\sigma_x$ be the simplex containing $x$, let $k_x\in \{1,\dots,n\}$ and $\alpha_x\in \pi$ such that $\tilde f_{k_x}(\tilde \sigma_x) = \alpha_x\tilde \sigma_x$, and let $o(f,x) \in \{+1,-1\}$ be the orientation with which $f$ maps $\sigma_x$ onto itself. Then we will have:
\[ T(f,\tilde f) = \sum_{x\in \Fix(f)} (-1)^{\dim \sigma_x} o(f,x)(\alpha_x,k_x). \]

The classical fixed point index can be computed in terms of the local orientation of the mapping. Proposition 3.2 of \cite{geog02} gives the following characterization: if $x$ is an isolated fixed point of $f$ inside a maximal simplex of dimension $d$, and $U_x$ is a neighborhood of $x$ containing no other fixed points, then $\ind(f,U_x) = (-1)^d$ if $f$ preserves the orientation near $x$, and $(-1)^{d+1}$ if $f$ reverses the orientation near $x$. Thus the factor of $(-1)^{\dim\sigma_x}o(f,x)$ above is exactly the fixed point index of $f_{k_x}$ at $x$, which equals the index of $f$ at $x$. Thus we have:
\[ T(f,\tilde f) = \sum_{x\in \Fix(f)} \ind(f,U_x) (\alpha_x,k_x). \]

By construction we have $x\in p\Fix(\alpha_x^{-1}\tilde f_{k_x})$, and thus $[(\alpha_x,k_x)] = [x]$. Thus we have:
\[ \rho(T(f,\tilde f)) = \sum_{x\in \Fix(f)} \ind(f,U_x) [x] = \RT(f,\tilde f) \]
by Theorem \ref{regularwecken}. 
\end{proof}

The same trace formula above also holds for non-regular maps, though the proof will require the following lemma. The proof of the lemma is a detailed algebraic calculation, which we postpone to the next section.
\begin{lem}\label{alglem}
Let $K$ be a complex with dimension $d$, let  $K'$ be a subdivision of $K$, and let $f,g:|K'| \to D_n(|K|)$ be two homotopic $n$-valued simplicial maps. Let $\tilde f = (\tilde f_1,\dots, \tilde f_n): |\tilde K'| \to F_n(|\tilde K|,\pi)$ be a lifting, and $\tilde g = (\tilde g_1,\dots, \tilde g_n)$ be the lifting of $g$ obtained by lifting the homotopy from $f$ to $g$ so that $\tilde f_k$ is homotopic to $\tilde g_k$ for each $k$. Then:
\begin{equation}
\sum_{q=0}^d \sum_{k=1}^n (-1)^q  \rho(\tr (\tilde f_k^q \circ \epsilon^q) \times \{k\}) = \sum_{q=0}^d \sum_{k=1}^n (-1)^q  \rho(\tr (\tilde g_k^q \circ \epsilon^q) \times \{k\}).
\end{equation}
\end{lem}

Theorem 4 of \cite{schi84a} is an $n$-valued simplicial approximation theorem, showing that any continuous $n$-valued map on polyhedra is is arbitrarily close to a regular $n$-valued simplicial map on some subdivision. This, together with homotopy invariance, allows us to drop the regularity assumption from Theorem \ref{regweckentracethm}

\begin{thm}\label{weckentracethm}
Let $K'$ be a subdivision of $K$, and let $f:|K'| \to D_n(|K|)$ be an $n$-valued simplicial map. Let $\tilde f = (\tilde f_1,\dots, \tilde f_n): |\tilde K'| \to F_n(|\tilde K|,\pi)$ be a lifting. Then:
\begin{equation}
\RT(f,\tilde f) = \sum_{q=0}^d \sum_{k=1}^n (-1)^q  \rho(\tr (\tilde f_k^q \circ \epsilon^q) \times \{k\}) \in \Z \Reid(\tilde f_\#) .
\end{equation}
\end{thm}
\begin{proof}
By Schirmer's approximation theorem, there is a regular simplicial map $g$ homotopic to $f$, and by homotopy invariance and Theorem \ref{regweckentracethm} we know that 
\[ \RT(f,\tilde f) = \RT(g,\tilde g) = \sum_{q=0}^d \sum_{k=1}^n (-1)^q  \rho(\tr (\tilde g_i^q \circ \epsilon^q) \times \{k\}), \]
and by Lemma \ref{alglem} this equals $\sum_{q=0}^d \sum_{k=1}^n (-1)^q  \rho(\tr (\tilde f_k^q \circ \epsilon^q) \times \{k\})$ as desired.
\end{proof}

\begin{exl}\label{circleexl}
As an example of Theorem \ref{weckentracethm}, we will compute the Reidemeister trace for a $2$-valued map on the circle $S^1$, viewing $S^1$ as real numbers read modulo 1. In \cite{brow06}, Brown showed that all $n$-valued maps on the circle are homotopic to a ``linear'' map of some degree $d$, defined by 
\[ f(t) = \left\{ \frac dn t, \frac dn t + \frac1n, \dots, \frac dn t + \frac{n-1}n\right\}. \]
The graph of such a linear map consists of $n$ parallel lines of slope $d/n$. For example the $2$-valued linear map of degree $5$ has graph that looks like:
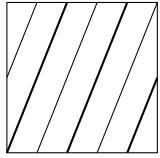
\tikzstyle{line1}=[thick]
\tikzstyle{line2}=[]
\[ 
\begin{tikzpicture}[scale=2]
\draw (0,0) rectangle (1,1);
\draw[line2] (0,1/2) -- (1/5, 1);
\draw[line1] (0,0) -- (2/5,1);
\draw[line2] (1/5,0) -- (3/5,1);
\draw[line1] (2/5,0) -- (4/5,1);
\draw[line2] (3/5,0) -- (1,1);
\draw[line1] (4/5,0) -- (1,1/2);
\end{tikzpicture}
 \]

To demonstrate the trace formula, we will compute the Reidemeister trace of the map above.  
Our first step must be to choose a triangulation $K$ and subdivision $K'$ such that we have a simplicial map $K' \to K$. There is not an obvious way to do this for the example above. Instead we will change the example pictured above by a small homotopy to obtain the $2$-valued map $f:S^1 \to D_2(S^1)$ of degree $5$ pictured in Figure \ref{circlefig}:

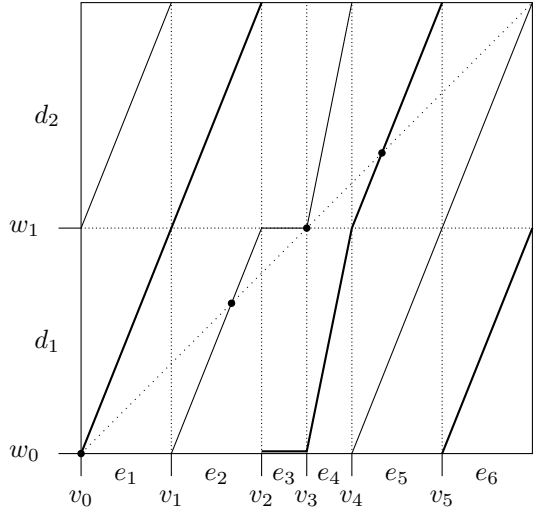
\begin{figure}
\newcommand{\tw}{.05}
\newcommand{\ylabf}{2.5}
\newcommand{\xlabf}{3}
\[ 
\begin{tikzpicture}[scale=6]
\draw (0,0) rectangle (1,1);
\draw[line2] (0,1/2) -- (1/5, 1);
\draw[line1] (0,0) -- (2/5,1);
\draw[line2] (1/5,0) -- (2/5,1/2) -- (1/2,1/2) -- (3/5,1);
\draw[line1] (2/5,.005) -- (1/2,.005) -- (3/5,1/2) -- (4/5,1);
\draw[line2] (3/5,0) -- (1,1);
\draw[line1] (4/5,0) -- (1,1/2);
\foreach \y in {0,1} {
  \draw (0,\y/2) -- (-\tw,\y/2);
  \node at (-\ylabf*\tw, \y/2) {$w_\y$};
  }
\foreach \y in {1,2} {
  \node at ({-(\ylabf-1)*\tw}, {(\y-1)/2+1/4}) {$d_\y$};
  }
\foreach \x in {0,...,2} {
 \draw (\x/5,0) -- (\x/5,-\tw);
 \node at (\x/5,{-(\xlabf-1)*\tw}) {$v_\x$};
 }
\draw (1/2,0) -- (1/2,-\tw);
\node at (1/2,{-(\xlabf-1)*\tw}) {$v_3$};
\foreach \x in {3,4} {
 \draw (\x/5,0) -- (\x/5,-\tw);
 }
\node at (3/5,{-(\xlabf-1)*\tw}) {$v_4$};
\node at (4/5,{-(\xlabf-1)*\tw}) {$v_5$};
\node at (.1,{-(\xlabf-2)*\tw}) {$e_1$};
\node at (.3,{-(\xlabf-2)*\tw}) {$e_2$};
\node at (.45,{-(\xlabf-2)*\tw}) {$e_3$};
\node at (.55,{-(\xlabf-2)*\tw}) {$e_4$};
\node at (.7,{-(\xlabf-2)*\tw}) {$e_5$};
\node at (.9,{-(\xlabf-2)*\tw}) {$e_6$};

\draw[ dotted] (0,0) -- (1,1);
\foreach \x in {0,1/3,1/2,2/3} {
  \filldraw (\x,\x) circle (.2pt);
  }
  
\foreach \x in {.2,.4,.5,.6,.8} {
 \draw[densely dotted] (\x,0) -- (\x,1);
 }
\draw[densely dotted] (0,.5) -- (1,.5);
\end{tikzpicture}
\]
\caption{The function $f$ of Example \ref{circleexl}. Thick lines are parts of the graph which lift to $\tilde f_1$, the thin lines lift to $\tilde f_2$. The marked points are the fixed points of $f$.\label{circlefig}}
\end{figure}

Since we are using the circle, our triangulations are uniquely specified by their vertex sets. In the codomain we are using vertex set $K_0 = \{w_0,w_1\} = \{0,1/2\}$, and in the domain we are using 
\[ K'_0 = \{v_0,\dots,v_5\} = \{0,1/5,2/5,1/2,3/5,4/5\}. \] 
We will denote the edges of $K$ as $K_1 = \{d_1,d_2\}$ as shown in the picture, and the edges of $K'$ as $K'_1=\{e_1,\dots,e_6\}$.

The universal cover of $S^1$ is $\R$, with projection map $p:\R \to S^1=[0,1)$ given by taking the non-integer part of a real number. As our lifting, we take $\tilde f: \R \to F_2(\R,\pi)$ given by $\tilde f = (\tilde f_1, \tilde f_2)$ where $\tilde f_1(0) = 0$ and $\tilde f_2(0)=1/2$. In the picture, the parts of the graph corresponding to $\tilde f_1$ are drawn with thick lines, and the parts corresponding to $\tilde f_2$ are lighter.

Lifting our triangulations $K$ and $K'$ gives triangulations $\tilde K$ and $\tilde K'$. As our $\Z\pi$-bases for $C_q(K)$ we will choose vertices $\tilde w_0 = 0$ and $\tilde w_1 = 1/2$ and edges $\tilde d_1 = [0,1/2]$ and $\tilde d_2 = [1/2,1]$. Similarly for $K'$ we make the natural choice of basis elements starting with $\tilde v_0 = 0$, $\tilde v_1 = 1/5$, etc.

Our fundamental group $\pi_1(S^1)$ is isomorphic to $\Z$, which we will write multiplicatively with generator $a$.

Now we can compute the various traces of \eqref{traceeq}. We have:
\begin{align*}
\tilde f^0_1\circ \epsilon^0(\tilde w_0) &= \tilde f^0_1(\tilde v_0) = \tilde w_0, \\
\tilde f^0_1\circ \epsilon^0(\tilde w_1) &= \tilde f_1^0(\tilde v_3) = a\tilde w_0.
\end{align*}
Thus we can express $\tilde f_1^0 \circ \epsilon_0$ as the matrix:
\[ \tilde f_1^0\circ \epsilon^0 = \begin{bmatrix} 1 & a \\ 0 & 0 \end{bmatrix} \]
and so $\tr(\tilde f_1^0\circ \ep^0) = 1$.

Similar computations give the matrix:
\[ \tilde f_2^0 \circ \ep^0 = \begin{bmatrix} 0 & 0 \\ 1 & a \end{bmatrix} \]
and so $\tr(\tilde f_2^0\circ \ep^0) = a$.

In dimension 1, we have:
\begin{align*}
\tilde f^1_1\circ \ep^1(\tilde d_1) &= \tilde f^1_1(\tilde e_1) + \tilde f^1_1(\tilde e_2) + \tilde f^1_1(\tilde e_3) = \tilde d_1 + \tilde d_2 + 0, \\
\tilde f^1_1\circ \ep^1(\tilde d_2) &= \tilde f^1_1(\tilde e_4) + \tilde f^1_1(\tilde e_5) + \tilde f^1_1(\tilde e_6) = a\tilde d_1 + a\tilde d_2 + a^2 \tilde d_1
\end{align*}
and so we have the matrix:
\[ \tilde f^1_1\circ \ep^1 = \begin{bmatrix} 1 & a+a^2 \\ 1 & a \end{bmatrix} \]
and so $\tr(\tilde f^1_1\circ \ep^1) = 1 + a$.

Similar computations show 
\[ \tilde f_2^1\circ \ep^1 = \begin{bmatrix} a & a^2 \\ 1 & a+a^2 \end{bmatrix} \]
and so $\tr(\tilde f_2^1\circ \ep^1) = 2a + a^2$.

Thus \eqref{traceeq} gives:
\begin{align*} 
\RT(f,\tilde f) &= [(1,1)] + [(a,2)] - [(1,1)] - [(a,1)] - 2[(a,2)] - [(a^2,2)] \\
&= -[(a,2)] - [(a,1)] - [(a^2,2)].
\end{align*}

Naturally we would like to know if the terms in the sum above can be combined or simplified any further. This will require some algebra to determine the structure of the set $\Reid(\tilde f_\#)$. In the case of the circle, this work appears in Section 5 of \cite{bdds20}, where it is shown that, for the linear $n$-valued map of degree $d$ with $n\neq d$, we have $[(a^i,k)] = [(a^l,j)]$ if and only if $di+k = dl+j \mod |d-n|$. In our example this criterion will demonstrate that all three terms above are distinct, and it will also allow us to ``simplify'' $[(a,2)]=[(1,2)]$ and $[(a^2,2)] = [(1,1)]$. Thus we conclude:
\[ \RT(f,\tilde f) = -[(1,1)] - [(1,2)] - [(a,1)], \]
and so in particular we have $L(f)=-3$ and $N(f)=3$. (In fact Brown has shown in \cite[Theorem 2]{brow07} that $N(f)=|L(f)|$ for any $n$-valued map on the circle.)

Rather than using the trace formula, we could directly apply Theorem \ref{regularwecken} to the example. Inspection of the graph in Figure \ref{circlefig} shows 4 fixed points, with $\Fix(f) = \{0,1/3,1/2,2/3\}$. The fixed point at 1/2 can be removed by a small homotopy and so it has fixed point index 0. The others each have fixed point index $-1$, and we observe that:
\[ \tilde f_1(0)=0, \qquad \tilde f_2(1/3) = 1/3, \qquad \tilde f_1(2/3) = a(2/3), \]
and so the fixed point classes are as follows:
\[ 0 \in p\Fix(\tilde f_1(0)) \qquad
1/3 \in p\Fix(\tilde f_2(1/3)) \qquad
2/3 \in p\Fix(a^{-1}\tilde f_1(2/3)) \]
Thus the Reidemeister class associated to the fixed point 0 is $[(1,1)]$, the Reidemeister class of $1/3$ is $[(1,2)]$, and the Reidemeister class of $2/3$ is $[(a,1)]$. As expected, Theorem \ref{regularwecken} gives:
\[ \RT(f,\tilde f) =-[(1,1)] - [(1,2)] - [(a,1)]. \]
\end{exl}

Following Brown's work we can compute the Reidemeister trace for any circle map. As in Example \ref{circleexl} we represent $\pi_1(S^1)$ multiplicatively as the group generated by $a\in \pi_1(S^1)$. 
\begin{thm}
Let $f:S^1\to D_n(S^1)$ be the $n$-valued map of degree $d$, and let $\tilde f$ be a lift. If $d=n$ then $\RT(f,\tilde f) = 0$. If $d>n$, then let 
\[ S = \{ (a^i,k) \in \pi\times \{1,\dots,n\} \mid 0\le -in-k+1 < d-n \}, \]
and we have
\[ \RT(f,\tilde f) = \sum_{(a^i,k)\in S} -[(a^i,k)]. \]
If $d<n$ then similarly let
\[ T = \{ (a^i,k) \in  \pi\times \{1,\dots,n\} \mid 0\ge -in-k+1 > d-n \}, \]
and we have
\[ \RT(f,\tilde f) = \sum_{(a^i,k)\in T} [(a^i,k)]. \]
\end{thm}
\begin{proof}
By Theorem 3.1 of \cite{brow06}, we may change by homotopy so that $f$ is the linear $n$-valued map of degree $d$, with lifting $\tilde f = (\tilde f_1,\dots,\tilde f_n)$ given by:
\[ \tilde f_k(t) = \frac dn t + \frac{k-1}n. \]
If $n=d$, then changing $f$ by homotopy to $f(x)+\epsilon$ for some small $\epsilon>0$ will remove all fixed points. Thus when $n=d$ by Theorem \ref{regularwecken} we will have $\RT(f,\tilde f) = 0$.

Now we consider the case where $n\neq d$. For any $i\in \Z$, we will compute $\Fix(a^{-i} \tilde f_k)$. We have $t\in \Fix(a^{-i} \tilde f_k)$ if and only if $\tilde f_k(t) - i = t$, which is equivalent to:
\[ t = \frac dn t + \frac{k-1}n - i,\]
and solving for $t$ gives:
\[ t = \frac{-in-l+1}{d-n}. \]
Thus we have $x\in p\Fix(a^{-i}\tilde f_k)$ if and only if $0\le x < 1$ and $x = \frac{-in-k+1}{d-n}$.

We will finish the proof in the case where $d>n$. The case where $d<n$ is similar. 

Since $d>n$, the graph of $f$ will consist of line segments having slope greater than 1, and so the index of each isolated fixed point will be $-1$. We will have $d-n>0$ and so we have a fixed point $x$ for every choice of $i$ and $k$ which makes $0 \le -in-k+1 < d$, and such a fixed point will be an element of $p\Fix(a^{-i}\tilde f_k)$, and thus $[x] = [(a^i,k)]$.
\end{proof}

Note that if $f$ is the $2$-valued circle map with degree 5, in the theorem above we have $S = \{(1,1),(a,1),(a,2) \}$, and thus
\[ \RT(f,\tilde f) = -[(1,1)] - [(a,1)] - [(a,2)], \]
agreeing with the results of Example \ref{circleexl}. 

\section{Proof of Lemma \ref{alglem}}\label{lemsec}
In this section we will prove Lemma \ref{alglem}, the homotopy-invariance of the trace formula for $\RT(f,\tilde f)$. As a warm-up, we will outline the proof used for the single-valued theory, following some exposition from \cite{geog02} (see also \cite{huss82} which follows these arguments in more generality). 

In the single-valued setting, we have the chain complex $C^q(\tilde X)$ considered as $\Z\pi$-modules with some fixed chosen bases, the boundary map $\pd$ and $\phi$-maps $f^q:C^q(\tilde X)\to C^q(\tilde X)$ satisfying $f^q(g\tilde \sigma) = \phi(g)f^q(\tilde \sigma)$ for $g\in \pi$ where $\phi:\pi\to \pi$ is the homomorphism induced by $f$ on the fundamental group of $X$. (To simplify the notation we write the map $\tilde f^q \circ \epsilon^q$ of Lemma \ref{alglem} simply as $f^q$.)

We also assume another $\phi$-map $g$, which is homotopic to $f$. This homotopy induces a chain homotopy $H^q:C^q(\tilde X) \to C^{q+1}(\tilde X)$ satisfying 
\begin{equation}\label{chainhtp} 
f^q - g^q = \pd^{q+1} \circ H^q + H^{q-1}\circ \pd^q.
\end{equation}
This $H^q$ also satisfies $H^q(g\tilde \sigma) = \phi(g)H^q(\tilde \sigma)$. 

Our goal is to prove that:
\[ \sum_q (-1)^q \rho(\tr f^q) = \sum_q (-1)^q \rho(\tr g^q). \]

We do this by showing that the right side of \eqref{chainhtp} equals zero when we compute the alternating sum of Reidemeister classes of traces. 

Let $(v_j)$ and $(w_j)$ be the chosen bases for $C^{q}(\tilde X)$ and $C^{q-1}(\tilde X)$ respectively. 
Let the entries of the matrix for $\pd^{q}$ be given by $[\pd^{q}]^{ij} = d^{ij}$, where
\[ \pd^{q}(v_j) = \sum_{l}d^{lj}w_l, \]
and let $[H^{q-1}]^{ij} = h^{ij}$, where
\[ H^{q-1}(w_j) = \sum_l h^{lj}v_l. \]
Then to compute the matrix for $H^{q-1} \circ \pd^{q}$, we have:
\begin{align*} 
H^{q-1} (\pd^{q} (v_j)) &= H^{q-1} \left(\sum_l d^{lj} w_l\right) = \sum_l \phi(d^{lj}) H^{q-1}(w_l) \\
&= \sum_l \sum_m \phi(d^{lj}) h^{ml}w_m 
\end{align*}
and so we have $[H^{q-1} \circ \pd^{q}]^{ij} = \sum_l \phi(d^{lj})h^{il}$, and thus 
\[ \tr(H^{q-1} \circ \pd^{q}) = \sum_i \sum_l \phi(d^{li})h^{il}= \sum_i\sum_l \phi(d^{il})h^{li}, \]
where the second equality is by the commutativity of the classical trace.

Since $\pd^{q-1}$ satisfies $\pd^{q-1}(g\sigma) = g\pd^{q-1}(\sigma)$, a similiar computation gives 
\[ \tr (\pd^{q}\circ H^{q-1})  = \sum_i \sum_l h^{il}d^{li}. \] 
Since $\phi(x^{-1})\phi(x)yx$, we have $\rho(\phi(x)y) = \rho(yx)$, and so
\begin{equation}\label{warmupeq}  
\rho(\tr(H^{q-1}\circ \pd^q)) = \rho(\tr (\pd^{q}\circ H^{q-1})) . 
\end{equation}

Thus \eqref{chainhtp} gives:
\begin{align*} 
\RT(f) - \RT(g) &= \sum_q (-1)^q \rho(\tr(\pd^{q+1} \circ H^q + H^{q-1}\circ \pd^q)) \\
&= \sum_q (-1)^q \rho(\tr(\pd^{q+1} \circ H^q)) + (-1)^q \rho(\tr(H^{q-1}\circ \pd^q))
\end{align*}
and the sum will telescope to 0, so $\RT(f)=\RT(g)$.

Our proof of Lemma \ref{alglem} will repeat the arguments above in the $n$-valued setting. The argument is the same, but the notation will be more cumbersome.

First we have a simple lemma about Reidemeister classes in the $n$-valued setting.
\begin{lem}\label{rclem}
Let $\tilde f_\# = (\phi_1,\dots,\phi_n;\sigma)$, and let $x,y\in \pi$ and $k\in \{1,\dots,n\}$. Then:
\[ [(yx,\sigma_x^{-1}(k))]_{\tilde f_\#} = [(\phi_k(x)y,k)]_{\tilde f_\#} \]
\end{lem}
\begin{proof}
This follows immediately from \eqref{tceq}, using $\gamma = x$.
\end{proof}

Now we are ready to prove Lemma \ref{alglem}. 
\begin{proof}[Proof of Lemma \ref{alglem}]

The homotopy from $f$ to $g$ induces $n$ chain homotopies $(H_1,\dots,H_n)$ where $H^q_k$ is a chain homotopy from $f^q_k$ to $g^q_k$. Repeating the argument from the warm-up, it will suffice to show (compare \eqref{warmupeq}) that 
\begin{equation}\label{treq}
\sum_{k=1}^n \rho(\tr(H_k^{q-1}\circ \pd^q)\times \{k\}) = \sum_{k=1}^n \rho(\tr (\pd^{q}\circ H_k^{q-1})\times \{k\}).
\end{equation}

Let $(v_j)$ and $(w_j)$ be the chosen bases for $C^{q}(\tilde X)$ and $C^{q-1}(\tilde X)$ respectively, and let the matrix for $\pd^{q}$ be given by $[\pd^{q}]^{ij} = d^{ij}$. Let $[H^{q-1}_k]^{ij} = h_k^{ij}$. Then following the warm-up we have:
\begin{align*} 
H_k^{q-1} (\pd^{q} (v_j)) &= H_k^{q-1} \left(\sum_l d^{lj} w_l\right) = \sum_l \phi_k(d^{lj}) H_{\sigma_{d^{lj}}^{-1}(k)}^{q-1}(w_l) \\
&= \sum_l \sum_m \phi_k(d^{lj}) h_{\sigma_{d^{lj}}^{-1}(k)}^{ml}w_m 
\end{align*}
and so
\begin{align*} 
\tr(H_k^{q-1} \circ \pd^{q}) &= \sum_i \sum_l \phi_k(d^{li}) h_{\sigma_{d^{li}}^{-1}(k)}^{il} = \sum_i \sum_l \phi_k(d^{il}) h_{\sigma_{d^{il}}^{-1}(k)}^{li}.
\end{align*}
Then we have
\begin{align*} 
\rho(\tr(H_k^{q-1} \circ \pd^{q})\times \{k\})&= \sum_i \sum_l [(\phi_k(d^{il}) h_{\sigma_{d^{il}}^{-1}(k)}^{li}, k)] \\
&= \sum_i \sum_l [(h_{\sigma_{d^{il}}^{-1}(k)}^{li} d^{il}, \sigma_{d^{il}}^{-1}(k))]
\end{align*}
where the last equality is by Lemma \ref{rclem}.

Again following the warm-up the situation for $\pd^{q}\circ H_k^{q-1}$ is easier, and we will have:
\[ \tr(\pd^{q}\circ H_k^{q-1}) = \sum_i \sum_l h_k^{li} d^{il}, \]
and so
\[ \rho(\tr(\pd^{q}\circ H_k^{q-1}) \times \{k\}) = \sum_i \sum_l [(h_k^{li} d^{il},k)]. \]

Summarizing the above, and summing over $k$, we have:
\begin{align*}
\sum_{k=1}^n \rho(\tr(H_k^{q-1} \circ \pd^{q})\times \{k\}) &= \sum_i \sum_l \sum_{k=1}^n [(h_{\sigma_{d^{il}}^{-1}(k)}^{li} d^{il}, \sigma_{d^{il}}^{-1}(k))] \\
&= \sum_i \sum_l \sum_{j=1}^n [(h_j^{li} d^{il}, j)] 
= \sum_{j=1}^n \rho(\tr(\pd^{q}\circ H_j^{q-1}) \times \{j\})
\end{align*}
which establishes \eqref{treq}. In the middle step we have reindexed the sum since $\sigma^{-1}_{d^{il}}$ is a permutation.
\end{proof}

\section{Properties of the Reidemeister trace}\label{propertiessec}
Fixed point invariants typically are well-behaved with respect to splittings of $n$-valued maps. We will show that when $f$ splits as $f = (f_1, \dots, f_n)$, then the Reidemeister trace of $f$ equals the sum of the Reidemeister traces of the single-valued maps $f_k$. 

For a single valued map $f_k$ with induced homomorphism $\phi_k$, let $\Reid_1(\tilde f_{k\#})$ be the set of (single-valued) Reidemeister classes of $f_k$, given by elements of $\pi$ modulo the equivalence: $[\alpha] = [\beta]$ if and only $\alpha = \phi_k(\gamma)^{-1}\beta\gamma$ for some $\gamma\in \pi$.

When $f$ splits, we will have $\tilde f_\# = (\phi_1,\dots,\phi_n;\id)$, where $\id$ is the identity permutation. In this case it is easy to verify that the map $\iota_k: \Reid_1(\tilde f_{k\#}) \to \Reid(\tilde f_\#)$ given by $\iota_k([\alpha]) = [(\alpha,k)]$ is well defined, and that we have a disjoint union: 
\[ \Reid(\tilde f_\#) = \iota_1(\Reid_1(\tilde f_{1\#})) \sqcup \dots \sqcup \iota_n(\Reid_1(\tilde f_{n\#})). \] 

\begin{thm}
Let $f:X\to D_n(X)$, and assume that $f$ splits into single-valued maps as $f = (f_1,\dots,f_n)$. Then, choosing liftings $\tilde f_k:\tilde X \to \tilde X$ and letting $\tilde f = (\tilde f_1,\dots,\tilde f_n)$, we have:
\[ \RT(f,\tilde f) = \sum_{k=1}^n \iota_k(\RT(f_k, \tilde f_k)), \]
where the terms on the right side denote the Reidemeister trace of a single-valued map.
\end{thm}
\begin{proof}
By the disjoint union above, we have:
\begin{align*}
\RT(f,\tilde f) &= \sum_{[(\alpha,k)] \in \Reid(\tilde f_\#)} \ind(f,U_{[(\alpha,k)]})[(\alpha,k)] \\
&= \sum_{k=1}^n \sum_{[\alpha]\in \Reid_1(\tilde f_{k\#})} \ind(f,U_{\iota_k([\alpha])}) \iota_k([\alpha]) \\
&= \sum_{k=1}^n \iota_k\left( \sum_{[\alpha]\in \Reid_1(\tilde f_{k\#})} \ind(f,U_{\iota_k([\alpha])}) [\alpha] \right)
\end{align*}
Since every fixed point of $f$ on $U_{\iota_k([\alpha])}$ is a fixed point of $f_k$, we have $\ind(f,U_{\iota_k([\alpha])}) = \ind(f_k,U_{\iota_k([\alpha])})$, and thus
\[ \RT(f,\tilde f) = 
\sum_{k=1}^n \iota_k\left( \sum_{[\alpha]\in \Reid_1(\tilde f_{k\#})} \ind(f,U_{\iota_k([\alpha])}) [\alpha] \right)
= \sum_{k=1}^n \iota_k(\RT(f_k,\tilde f_k)) \]
as desired.
\end{proof}

Next we consider the role of the choice of lifting $\tilde f$ in $\RT(f,\tilde f)$. We show that alternative choices of the lifting will change the value of $\RT(f,\tilde f)$ in predictable ways. Given one lifting $\tilde f$, any other choice of lifting can be written as $\Phi\tilde f$ for some $\Phi \in \pi^n \rtimes \Sigma_n$. 

Work in \cite[page 7--8]{bdds20} describes the form of the induced homomorphism $(\Phi f)_\#$ in terms of $\Phi$ and $\tilde f_\#$, showing that:
\[ (\Phi f)_\#(\gamma) = \Phi \cdot \tilde f_\#(\gamma) \cdot \Phi^{-1}. \]
Let $\tau_\Phi$ denote this conjugation by $\Phi$, so $(\Phi f)_\# = \tau_\Phi \tilde f_\#$. 


\begin{prop}\label{muprop}
Let $f:X\to D_n(X)$ be an $n$-valued map and let $\Phi = (\delta_1,\dots,\delta_n;\epsilon) \in \pi^n \rtimes \Sigma_n$. Define $\mu_\Phi: \pi \times \{1,\dots,n\} \to \pi \times \{1,\dots,n\}$ by:
\[\mu_\Phi(\alpha,k) = (\delta_{k}^{-1}\alpha,\epsilon^{-1}(k)) \]
Then $\mu_\Phi$ induces a well-defined bijection $\mu_\Phi: \Reid(\tau_\Phi \tilde f_\#) \to \Reid(\tilde f_\#)$. 
\end{prop}
\begin{proof}
Let $[(\alpha,k)]_{\tau_\Phi \tilde f_\#} = [(\beta,j)]_{\tau_\Phi \tilde f_\#} \in \Reid(\tau_\Phi \tilde f_\#)$, and we will show that: 
\[ [(\delta_k^{-1}\alpha, \epsilon^{-1}(k))]_{\tilde f_\#} = [(\delta_j^{-1}\beta,\epsilon^{-1}(j))]_{\tilde f_\#}. \] That is, we must show that there is some $\gamma\in G$ with: 
\begin{equation}\label{mueq}
\sigma_\gamma(\ep^{-1}(k)) = \ep^{-1}(j) \text{ and } \delta_k^{-1} \alpha  = \phi_{\ep^{-1}(j)}(\gamma^{-1}) \delta_j^{-1} \beta \gamma
\end{equation}
where $\tilde f_\#(\gamma) = (\phi_1(\gamma),\dots,\phi_n(\gamma);\sigma_\gamma) \in \pi^n\rtimes \{1,\dots,n\}$.

In coordinates, the homomorphism $\tau_\Phi \tilde f_\#$ can be computed as:
\begin{align*}
\tau_\Phi \tilde f_\#(\gamma) &= \Phi \cdot \tilde f_\#(\gamma)\cdot \Phi^{-1} \\
&=  (\delta_1,\dots,\delta_n;\epsilon) \cdot (\phi_1(\gamma),\dots,\phi_n(\gamma);\sigma_\gamma) \cdot (\delta_1,\dots,\delta_n;\epsilon)^{-1} \\
&=  (\delta_1,\dots,\delta_n;\epsilon) \cdot (\phi_1(\gamma),\dots,\phi_n(\gamma);\sigma_\gamma) \cdot (\delta_{\epsilon(1)}^{-1},\dots,\delta_{\epsilon(n)}^{-1}; \ep^{-1}) \\
&=  (\delta_1,\dots,\delta_n;\epsilon) \cdot (\phi_1(\gamma)\delta^{-1}_{\sigma^{-1}_\gamma(\ep(1))}, \dots, \phi_n(\gamma)\delta^{-1}_{\sigma^{-1}_\gamma(\ep(n))}; \sigma_\gamma\circ \ep^{-1}) \\
&= (\delta_1 \phi_{\ep^{-1}(1)}(\gamma)\delta^{-1}_{\ep^{-1}(\sigma^{-1}_\gamma(\ep(1)))}, \dots, 
\delta_n \phi_{\ep^{-1}(n)}(\gamma)\delta^{-1}_{\ep^{-1}(\sigma^{-1}_\gamma(\ep(n)))}; \ep\circ \sigma_\gamma\circ \ep^{-1}).
\end{align*}

Thus our assumption that  $[(\alpha,k)]_{\tau_\Phi \tilde f_\#} = [(\beta,j)]_{\tau_\Phi \tilde f_\#}$ means that there is some $\gamma$ with $\ep\circ \sigma_\gamma\circ \ep^{-1}(k) = j$, and 
\[ \alpha =  (\delta_j\phi_{\ep^{-1}(j)}(\gamma)\delta^{-1}_{\ep^{-1}(\sigma^{-1}_\gamma(\ep(j)))})^{-1}  \beta \gamma. \]

The fact that $\ep\circ \sigma_\gamma\circ \ep^{-1}(k) = j$ establishes the first part of \eqref{mueq}, and allows us to simplify the above to:
\begin{align*}
\alpha &= (\delta_j\phi_{\ep^{-1}(j)}(\gamma)\delta^{-1}_k)^{-1} \beta \gamma 
= \delta_k \phi_{\ep^{-1}(j)}(\gamma)^{-1}\delta^{-1}_j \beta \gamma 
\end{align*}
which establishes the second part of \eqref{mueq}.

Thus we have a well-defined function $\mu_\Phi: \Reid(\tau_\Phi \tilde f_\#) \to \Reid(\tilde f_\#)$. To show this is a bijection, we observe that for $A,B \in \pi^n\rtimes \Sigma_n$, a calculation will show that $\mu_A \circ \mu_B = \mu_{BA}$. Thus $\mu_{\Phi^{-1}}$ is an inverse for $\mu_\Phi$, and so $\mu_\Phi$ is a bijection. 
\end{proof}

Because of the theorem above we will linearize and consider $\mu_\Phi$ as a map $\mu_\Phi: \Z\Reid(\tau_\Phi \tilde f_\#) \to \Z\Reid(\tilde f_\#)$. The following theorem gives the relationship between $\RT(f,\tilde f)$ and $\RT(f,\Phi\tilde f)$.

\begin{thm}\label{muthm}
Let $f:X\to D_n(X)$ be an $n$-valued map with a lifting $\tilde f: \tilde X \to F_n(\tilde X,\pi)$, and let $\Phi = (\delta_1,\dots,\delta_n;\epsilon) \in \pi^n \rtimes \Sigma_n$. Then:
\[ \mu_\Phi(\RT(f,\Phi \tilde f)) = \RT(f,\tilde f). \]
\end{thm}
\begin{proof}
Applying $\mu_\Phi$ to the formula of Theorem \ref{weckentracethm} we have:
\begin{align*}
\mu_\Phi(\RT(f,\Phi \tilde f)) &= \sum_{[(\alpha,k)] \in \Reid(\tau_\Phi \tilde f_\#)} \ind(f,p\Fix \alpha^{-1} (\Phi\tilde f)_{k}) \mu_\Phi([(\alpha,k)]_{\tau_\Phi \tilde f_\#}) \\
&= \sum_{[(\alpha,k)] \in \Reid(\tau_\Phi \tilde f_\#)} \ind(f,p\Fix \alpha^{-1} (\Phi\tilde f)_{k}) [(\delta_k^{-1}\alpha, \ep^{-1}(k))]_{\tilde f_\#}.
\end{align*}
where $(\Phi\tilde f)_k$ indicates the $k$th coordinate of $\Phi\tilde f: \tilde X \to F_n(X,\pi)$. (For convenience we are writing $\ind(f,p\Fix \alpha^{-1} (\Phi\tilde f)_{k})$ to mean $\ind(f,U)$ where $U$ is some small neighborhood of $p\Fix \alpha^{-1} (\Phi\tilde f)_{k}$ containing no other fixed points.)

We can compute this coordinate as follows:
\begin{align*}
\Phi\tilde f(\tilde x) &= (\delta_1,\dots,\delta_n;\epsilon) \cdot (\tilde f_1(\tilde x),\dots,\tilde f_n(\tilde x)) \\
&=(\delta_1 \tilde f_{\ep^{-1}(1)}(\tilde x), \dots, \delta_n \tilde f_{\ep^{-1}(n)}(\tilde x)),
\end{align*}
and so we have $(\Phi \tilde f)_k = \delta_k\tilde f_{\ep^{-1}(k)}$. Thus 
\[ \mu_\Phi(\RT(f,\Phi\tilde f)) = \sum_{[(\alpha,k)] \in \Reid(\tau_\Phi \tilde f_\#)} \ind(f,p\Fix \alpha^{-1} \delta_k \tilde f_{\ep^{-1}(k)}) [(\delta_k^{-1}\alpha, \ep^{-1}(k))]_{\tilde f_\#}. \]
Since $\mu_\Phi: \Reid(\tau_\Phi \tilde f_\#) \to \Reid(\tilde f_\#)$ is a bijection, we may reindex the sum, replacing $[(\delta_k^{-1}\alpha, \ep^{-1}(k))]_{\tilde f_\#}$ with $[(\beta,j)]_{\tilde f_\#}$. Thus the above becomes
\[ \mu_\Phi(\RT(f,\Phi\tilde f)) = \sum_{[(\beta,j)]\in \Reid(\tilde f_\#)} \ind(f,p\Fix \beta^{-1} \tilde f_j) [(\beta,j)]_{\tilde f_\#} = \RT(f,\tilde f). \qedhere
\]
\end{proof}

%
\section{A local Reidemeister trace, and an averaging formula}\label{localsec}

Like the fixed point index, the Reidemeister trace can be defined locally on some subset $U\subset X$ by considering only those fixed points appearing in $U$. An axiomatic approach to the local Reidemeister trace of a single-valued map was given in \cite{stae09b}. For a map $f:U\to D_n(X)$ where $U\subset X$ is open, we say the pair $(f,U)$ is \emph{admissible} when $\Fix(f)\cap U$ is compact. 

\begin{defn}\label{localRTdef}
Let $f:X\to D_n(X)$ be a map and $(f,U)$ be admissible such that $\Fix(f)\cap U$ is finite. Let $\tilde f: \tilde X \to F_n(\tilde X,\pi)$ be a lifting of $f$. Then:
\[ \RT(f,\tilde f,U) = \sum_{[(\alpha,k)]\in \Reid(\tilde f_\#)} \ind(f,U_{[(\alpha,k)]}\cap U) [(\alpha,k)]_{\tilde f} \in \Z\Reid(\tilde f_\#), \]
where $U_{[(\alpha,k)]}$ is a neighborhood of $p\Fix(\alpha^{-1}\tilde f_k)$ containing no other fixed points of $f$.
\end{defn}

Clearly from the definition we have $\RT(f,\tilde f,X) = \RT(f,\tilde f)$. This local Reidemeister trace also obeys familiar homotopy and additivity properties. 

If $H$ is a homotopy between two maps $f,g:U\to D_n(X)$ and $F(H) = \{(x,t) \in U\times [0,1] \mid x \in H(x,t) \}$ is compact, then we say $H$ is an \emph{admissible homotopy} of $(f,U)$ to $(g,U)$. As in Theorem \ref{htp}, we have:
\begin{thm}(Homotopy property)
For some open subset $U\subset X$, let $(f,U)$ and $(g,U)$ be admissibly homotopic by the admissible homotopy $H:X\times [0,1] \to D_n(X)$. 
Let $\tilde f$ be a lifting of $f$, and let $\tilde g$ be the lifting of $g$ obtained by lifting the homotopy, starting at $\tilde f$. Then, identifying $\Z\Reid(\tilde f_\#)$ with $\Z\Reid(g_\#)$, we have:
\[ \RT(f,\tilde f, U) = \RT(g,\tilde g,U). \]
\end{thm}

We also easily obtain excision and additivity properties:
\begin{thm}(Excision property)
If $(f,U)$ is admissable and $(f,V)$ is also admissible with $V \subset U$ and $\Fix(f)\cap V = \Fix(f)\cap U$, then $\RT(f,\tilde f,U) = \RT(f,\tilde f, V)$.
\end{thm}
\begin{proof}
In Definition \ref{localRTdef}, we will have $U_{[(\alpha,k)]} \cap U \subset U_{[(\alpha,k)]} \cap V$, and by the excision property of the index we will have $\ind(f,U_{[(\alpha,k)]} \cap U) = \ind(U_{[(\alpha,k)]} \cap V)$. Thus $\RT(f,\tilde f,U) = \RT(f,\tilde f, V)$.
\end{proof}

\begin{thm}(Additivity property)
Let $(f,U)$ be admissible and $\tilde f$ be a lifting of $f$. Let $U_1, U_2 \subset U$ be disjoint open sets with $\Fix(f)\cap U \subset U_1 \cup U_2$. Then:
\[ \RT(f,\tilde f,U) = \RT(f,\tilde f,U_1) + \RT(f,\tilde f,U_2). \]
\end{thm}
\begin{proof}
By the excision property we have $\RT(f,\tilde f, U) = \RT(f,\tilde f, U_1\cup U_2)$, and then:
\begin{align*}
\RT(f,\tilde f, U) &= 
\sum_{[(\alpha,k)]\in \Reid(\tilde f)} \ind(f,U_{[(\alpha,k)]}\cap (U_1\cup U_2)) [(\alpha,k)]_{\tilde f} \\
&= \RT(f,\tilde f,U_1) + \RT(f,\tilde f,U_2)
\end{align*}
by the additivity property of the index.
\end{proof}

The following property follows immediately from Definition \ref{localRTdef}, and specifies the value of the Reidemeister trace when $U$ contains only a single fixed point.
\begin{thm}(Isolated fixed point property)
Let $(f,U)$ be admissible such that $\Fix(f)\cap U =\{x\}$ is a single point. If $\tilde f = (\tilde f_1,\dots,\tilde f_n)$ is a lift with $x \in p\Fix(\alpha^{-1}\tilde f_k)$, then:
\[ \RT(f,\tilde f, U) = \ind(f,U)[(\alpha,k)]. \]
\end{thm}
\begin{proof}
Since $f$ has only one fixed point in $U$, the sum of Definition \ref{localRTdef} consists of a single term:
\[ \RT(f,\tilde f,U) = \ind(f,U_x)[(\alpha, k)] \]
where $U_x\subset U$ is a small neighborhood containing $x$. By the excision property, the above equals $\ind(f,U)[(\alpha,k)]$ as desired.
\end{proof}

We can easily obtain a uniqueness result in the spirit of the main theorem of \cite{stae09b}. For a space $X$, let $\adm(X)$ be the set of all triples $(f,\tilde f, U)$ where $f:X\to D_n(X)$ is an $n$-valued map and $\tilde f: \tilde X \to F_n(\tilde X,\pi)$ is a lifting and $U\subset X$ such that $(f,U)$ is admissible. 

\begin{thm}\label{RTuniqueness}
There is at most one function $\tau$ with domain $\adm(X)$ such that $\tau(f,\tilde f,U) \in \Z\Reid(\tilde f_\#)$, which satisfies the following properties:
\begin{itemize}
\item (Additivity) If $(f,\tilde f,U)\in \adm(X)$ and $U_1$ and $U_2$ are disjoint open subsets of $U$ with $\Fix(f)\cap U \subset U_1\cup U_2$, then:
\[ \tau(f,\tilde f,U) = \tau(f,\tilde f,U_1) + \tau(f,\tilde f,U_2). \]
\item (Homotopy) If $(f,\tilde f,U),(g,\tilde g,U)\in \adm(X)$ such that $(f,U)$ is admissibly homotopic to $(g,U)$ by a homotopy which lifts to a homotopy of $\tilde f$ to $\tilde g$, then 
\[ \tau(f,\tilde f,U) = \tau(g,\tilde g,U). \]
\item (Isolated fixed point) If $(f,\tilde f,U)\in \adm(X)$ and $f$ has a single fixed point $x$ on $U$ with $x \in p\Fix(\alpha^{-1}\tilde f_k(x))$, then:
\[ \tau(f,\tilde f, U) = \ind(f,U)[(\alpha,k)]. \]
\end{itemize}
\end{thm}
\begin{proof}
Let $\tau$ satisfy the three properties above. We must show that these three properties suffice to completely determine the value of $\tau(f,\tilde f,U)$.

By Theorem 4 of \cite{schi84a} there is a regular $n$-valued map $g$ where $(f,U)$ is admissibly homotopic to $(g,U)$. Letting $\tilde g$ be the lifting of $g$ obtained by lifting the admissible homotopy from $f$ to $g$, we will have $\tau(f,\tilde f, U) = \tau(g,\tilde g,U)$. 

Since $g$ is regular there is an open subset $V = V_1\sqcup \dots \sqcup V_m \subset U$ consisting of a disjoint union of small neighborhoods around each fixed point of $f$, by additivity and excision we have:
\[ \tau(f,\tilde f, V) = \sum_{i=1}^m \tau(g,\tilde g, V_i). \]
Let $x_i\in V_i$ be the fixed point of $g$ in $V_i$, and assume $x_i \in p\Fix(\alpha_i^{-1}\tilde g_{k_i})$. Then by the isolated fixed point property we have 
\[ \tau(f,\tilde f,V) = \sum_{i=1}^m \ind(g,V_i)[(\alpha_i^{-1},k_i)]. \]

Therefore the value of $\tau$ can be computed fully using only the three properties. Thus there can be at most one function satisfying the three properties.
\end{proof}

Since we already know the Reidemeister trace satisfies these three properties, we obtain:
\begin{cor}
Any function satisfying the three properties of Theorem \ref{RTuniqueness} is the local Reidemeister trace.
\end{cor}

Now we can state and prove our averaging formula for the Reidemeister trace. In the single-valued theory, an averaging formula for the Lefshetz number appeared in the 1980s in \cite{jian83}. Various limited averaging results for the Nielsen number appeared in \cite{kl07}, which make some specific assumptions on the maps used. Nevertheless a full averaging formula (without making these assumptions) was demonstrated for the Reidmeister trace in \cite{lz15}. This averaging formula was extended to Nielsen coincidence theory in \cite{ls17}. Following an approach from \cite{ls17}, we prove our averaging formula using the axiomatic approach above.

Let $q:\bar X \to X$ be a finite covering. This covering corresponds to a normal subgroup $G < \pi$ with finite quotient $\pi/G$. Let $f:X\to D_n(X)$ have some lifting $\bar f:\bar X \to D_n(\bar X)$. The induced homomorphisms for $f$ and $\bar f$ take the form $\tilde f_\#:\pi \to \pi^n \rtimes \Sigma_n$ and $\bar \tilde f_\#: G \to G^n \rtimes \Sigma_n$. The inclusion map $G\to \pi$ induces a function of Reidemeister sets 
$\iota: \Reid(\bar f_\#) \to \Reid(\tilde f_\#)$. When working with other homomorphisms $\tau_\Phi \tilde f_\#$ and $\tau_\Phi \bar f_\#$, we denote the inclusion by:
\[ \iota_\Phi: \Reid(\tau_\Phi \bar f_\#) \to \Reid(\tau_\Phi \tilde f_\#). \]

Given $\beta \in \pi$, let $\beta^n \in \pi^n \rtimes \Sigma_n$ be given by $\beta^n = (\beta, \dots, \beta; 1)$. Similarly for $\bar \beta \in G$, we write $\bar \beta^n = (\bar \beta, \dots, \bar \beta; 1) \in G^n \rtimes \Sigma_n$. 

\begin{thm}
Let $q:\bar X \to X$ be a finite covering with covering group $G$. Let $f:X\to D_n(X)$ be an $n$-valued map and $\tilde f:\tilde X \to F_n(\tilde X,\pi)$ be a lifting, with $\bar f:\bar X \to D_n(\bar X)$  the map induced by $\tilde f$. Then:
\[ \RT(f,\tilde f,U) = \frac1{[\pi:G]} \sum_{\bar \beta \in \pi/G} \mu_{\beta^n} \circ \iota_{\beta^n} (\RT(\bar \beta^n \bar f, \beta^n \tilde f, \bar U)), \]
where $\beta\in \pi$ is some element projecting to $\bar \beta \in \pi/G$, and $\bar U = q^{-1}(U)$. Above, the function $\bar \beta^n \bar f: \bar X \to D_n(\bar X)$ is the $n$-valued map obtained by translating each value of the $n$-valued map $\bar f$ by the covering transformation $\bar \beta: \bar X \to \bar X$. 
\end{thm}
\begin{proof}
Let 
\[ \tau(f,\tilde f, U) = \frac1{[\pi:G]} \sum_{\bar \beta \in \pi/G} \mu_{\beta^n} \circ \iota_{\beta^n} (\RT(\bar \beta^n \bar f, \beta^n \tilde f, \bar U)).  \]

First we will show that $\tau$ is well-defined, that is, independent of the element $\beta \in \pi$ which projects to $\bar \beta \in \pi/G$. Let $\gamma \beta,\beta\in \pi$ be two alternative choices of elements which both project to some $\bar \beta \in \pi/G$ (so $\gamma \in G$), and we will show that:
\begin{equation}\label{welldefeq}
\mu_{\beta^n}(\iota_{\beta^n}(\RT(\bar \beta^n\bar f, \beta^n\tilde f,\bar U))) = \mu_{(\gamma\beta)^n}(\iota_{(\gamma\beta)^n}(\RT(\bar \beta^n\bar f,(\gamma \beta)^n\tilde f,\bar U)))
\end{equation}

Let $[(\alpha,k)]\in \Reid(\tau_{(\gamma\beta)^n} \bar f_\#)$, and we have:
\begin{align*}
\mu_{(\gamma\beta)^n}(\iota_{(\gamma\beta)^n}([(\alpha,k)]_{\tau_{(\gamma\beta)^n}\bar  f_\#})) &= \mu_{(\gamma\beta)^n}([(\alpha,k)]_{\tau_{(\gamma\beta)^n} \tilde f_\#}) \\
&= [(\beta^{-1}\gamma^{-1}\alpha,k)]_{\tilde f_\#} = \mu_{\beta^n}([(\gamma^{-1}\alpha,k)]_{\tau_{\beta^n}\tilde f_\#}) \\
&= \mu_{\beta^n}(\iota_{\beta^n}([(\gamma^{-1}\alpha, k)]_{\tau_{\beta^n}\bar  f_\#})) \\
&= \mu_{\beta^n}(\iota_{\beta^n}(\mu_{\gamma^n}([(\alpha,k)]_{\tau_{(\gamma\beta)^n}\bar  f_\#})),
\end{align*}
and thus we have $\mu_{(\gamma\beta)^n}\circ \iota_{(\gamma\beta)^n} = \mu_{\beta^n} \circ \iota_{\beta^n} \circ \mu_{\gamma^n}$.

Now, starting with the right side of \eqref{welldefeq}, we have:
\begin{align*}
\mu_{(\gamma\beta)^n}(\iota_{(\gamma\beta)^n}(\RT(\bar \beta^n\bar f,(\gamma \beta)^n\tilde f,\bar U)) &= \mu_{\beta^n}(\iota_{\beta^n}(\mu_{\gamma^n}(\RT(\bar \beta^n\bar f,(\gamma \beta)^n\tilde f,\bar U)) \\
&= \mu_{\beta^n}( \iota_{\beta^n}(\RT(\bar \beta^n\bar f,\beta^n\tilde f,\bar U)))
\end{align*}
by Theorem \ref{muthm}, establishing \eqref{welldefeq}.

So $\tau$ is a well-defined function which takes a triple $(f,\tilde f, U)$ and gives an element of $\Reid(\tilde f_\#)$. To show that $\tau$ is the Reidemeister trace, it will suffice to show that $\tau$ satisfies the three properties of Theorem \ref{RTuniqueness}. 

For the additivity property, let $\Fix(f)\cap U \subset U_1\sqcup U_2$. Then, letting $\bar U_i = q^{-1}(U_i)$, we have $\Fix(\bar \beta f) \cap \bar U \subset \bar U_1 \sqcup \bar U_2$ for each $\bar \beta$, and so by the additivity property of $\RT$ we have:
\begin{align*} 
\tau(f,\tilde f,U) &= \frac{1}{[\pi:G]} \sum_{\bar \beta \in \pi/g} \mu_{\beta^n}\circ \iota_{\beta^n}(\RT(\bar \beta^n \bar f, \beta^n \tilde f, \bar U))\\
&= \frac{1}{[\pi:G]} \sum_{\bar \beta \in \pi/g} \mu_{\beta^n}\circ \iota_{\beta^n}(\RT(\bar \beta^n \bar f, \beta^n \tilde f, \bar U_1) + \RT(\bar \beta^n \bar f, \beta^n \tilde f, \bar U_2)) \\ 
&= \tau(f,\tilde f, U_1) + \tau(f,\tilde f,U_2) 
\end{align*}
as desired.

For the homotopy property, note that any admissible homotopy of $(f,U)$ to some $(g,U)$ will naturally lift to an admissible homotopy of $(\bar \beta \bar f,\bar U)$ to $(\bar \beta \bar g,\bar U)$, and thus $\tau(f,\tilde f, U) = \tau(g,\tilde g,U)$ by the homotopy property of $\RT$. 

For the isolated fixed point property, let $\Fix(f)\cap U=\{x\}$ with $x\in p\Fix(\alpha^{-1}\tilde f_k)$. Then there is some $\tilde x \in p^{-1}(x)$ with $\tilde f_k(\tilde x) = \alpha \tilde x$. Then for each $\bar \beta$, we have $\Fix(\bar \beta^n \bar f\cap \bar U) = \{\bar x\}$ with $\beta\tilde f_k(\tilde x) = \beta \alpha\tilde x$. 
Then the isolated fixed point property for the Reidemeister trace gives:
\[ \RT(\bar \beta^n \bar f, \beta^n\tilde f,\bar U) = \ind(\bar \beta^n\bar f, \bar U)[( \beta\alpha, k)]_{\tau_{\bar \beta} \bar f} \]
Thus we obtain:
\begin{align*}
\tau(f,\tilde f, U) &= \frac1{[\pi:G]} \sum_{\bar \beta \in \pi/G} \mu_{\beta^n} \circ \iota_{\beta^n} (\ind(\bar\beta^n \bar f, \bar U)[(\beta\alpha, k)]_{\tau_{\beta^n} \tilde f_\#}) \\
&= \frac1{[\pi:G]} \sum_{\bar \beta \in \pi/G} \ind(\bar\beta^n \bar f, \bar U) \mu_{\beta^n} \circ \iota_{\beta^n} ([(\beta\alpha, k)]_{\tau_{\beta^n} \tilde f_\#}) \\
&= \left( \frac1{[\pi:G]} \sum_{\bar \beta \in \pi/G} \ind(\bar\beta^n \bar f, \bar U) \right) [(\alpha, k)]_{\tilde f_\#} \\
&= \ind(f,U)[(\alpha,k)]_{\tilde f_\#},
\end{align*}
where the last equality is by the index averaging result in \cite{stae18}.
So $\tau$ satisfies the isolated fixed point property.
\end{proof}

By taking $U=X$, we immediately obtain an averaging formula for the (non-local) Reidemeister trace:
\begin{cor}
With notation as in the theorem above, we have:
\[ \RT(f,\tilde f) = \frac1{[\pi:G]} \sum_{\bar \beta \in \pi/G} \mu_{\beta^n} \circ \iota_{\beta^n} (\RT(\bar \beta^n \bar f, \beta^n \tilde f)). \]
\end{cor}

\bibliographystyle{hplain}

\end{document}